\documentclass[12pt]{amsart} 
\usepackage{amsmath, amsthm, amssymb, hyperref, color}
\usepackage[shortlabels]{enumitem}
\usepackage{graphicx}
\usepackage[all]{xypic}
\usepackage{makecell}
\oddsidemargin \evensidemargin
\newtheorem{theorem}{Theorem}

\newtheorem{lemma}[theorem]{Lemma}
\newtheorem{corollary}[theorem]{Corollary}

\newtheorem{remark}[theorem]{Remark}
\newtheorem{example}[theorem]{Example}
\newtheorem{conjecture}[theorem]{Conjecture}
\DeclareMathOperator{\Ima}{Im}

\newcommand{\RR}{\mathbb{R}}

\title{\textbf{Maximum Number of Modes of Gaussian Mixtures}}
\author[C.~Am\'endola]{Carlos Am\'endola}
\address{Carlos Am\'endola \\ Technische Universit\"at M\"unchen \\ Germany}
\email{carlos.amendola@tum.de}
\author[A.~Engstr\"om]{Alexander Engstr\"om}
\address{Alexander Engstr\"om \\ Aalto University, Helsinki \\ Finland}
\email{alexander.engstrom@aalto.fi}
\author[C.~Haase]{Christian Haase}
\address{Christian Haase \\ Freie Universit\"at Berlin \\ Germany}
\email{haase@math.fu-berlin.de}

\begin{document}
\maketitle
\begin{abstract}
\noindent Gaussian mixture models are widely used in Statistics. A
fundamental aspect of these distributions is the study of the local
maxima of the density, or modes. In particular, it is not known how many modes a mixture of $k$ Gaussians in $d$ dimensions can have. We give a brief account of this problem's
history. Then, we give improved lower bounds and the first upper bound on the maximum
number of modes, provided it is finite.

\end{abstract}
\section{Introduction}
The $d$-dimensional Gaussian distribution $N(\mu, \Sigma)$ can be
defined by its probability density function
\begin{equation}
\label{eq:gaussian}
\phi(x)=\frac{1}{\sqrt{ \det(2 \pi
    \Sigma)}}\mathrm{e}^{-\frac{1}{2}(x-\mu)^T \Sigma^{-1} (x-\mu)}.
\end{equation}
where $\mu \in \mathbb{R}^d$ is the mean vector and the symmetric
positive definite $d \times d$ matrix $\Sigma$ is the covariance
matrix. This density has a unique absolute maximum at $x=\mu$.
Now consider a mixture distribution $X$ consisting of $k$ Gaussian
components $X_i \sim N(\mu_i,\Sigma_i)$ and mixture weights $\alpha_i$
for $i=1,\ldots, k$, so that:
\begin{equation}
\Phi_X(x)= \sum_{i=1}^k \alpha_i \phi_{X_i}(x).
\end{equation}
This is again a probability density function given that $\alpha_i \geq
0$ and $\alpha_1 + \alpha_2 + \cdots + \alpha_k = 1$. In other words,
the density of a Gaussian mixture is a convex combination of Gaussian
densities. Such mixtures can exhibit quite complex behavior even for a
small number $k$ of components. This is a feature that makes them
attractive for modeling in applications. 

A fundamental property of a probability density function is the number of modes, i.e.\ local maxima, that it possesses. For Gaussian mixtures, this is especially relevant in applications such as clustering \cite[p. 383]{Cl}. For example, the \textit{mean shift algorithm} converges if there are only finitely many critical points \cite{W}.

We will be interested in the maximal number $m(d,k)$ of local maxima
for $d$-dimensional Gaussian mixtures with $k$ components. Shockingly,
it is not known whether this maximal number is always finite for
general Gaussian mixtures. On the other hand, we stress that the number of modes is a
property of a Gaussian mixture density with fixed parameters and no
sample involved; it should not be confused with the number of local
maxima of the likelihood function of a Gaussian mixture model (a
relevant but different question, see \cite{ADS} and \cite{JOR}).

\begin{figure}[t]
  \centering
 \includegraphics[scale=0.49]{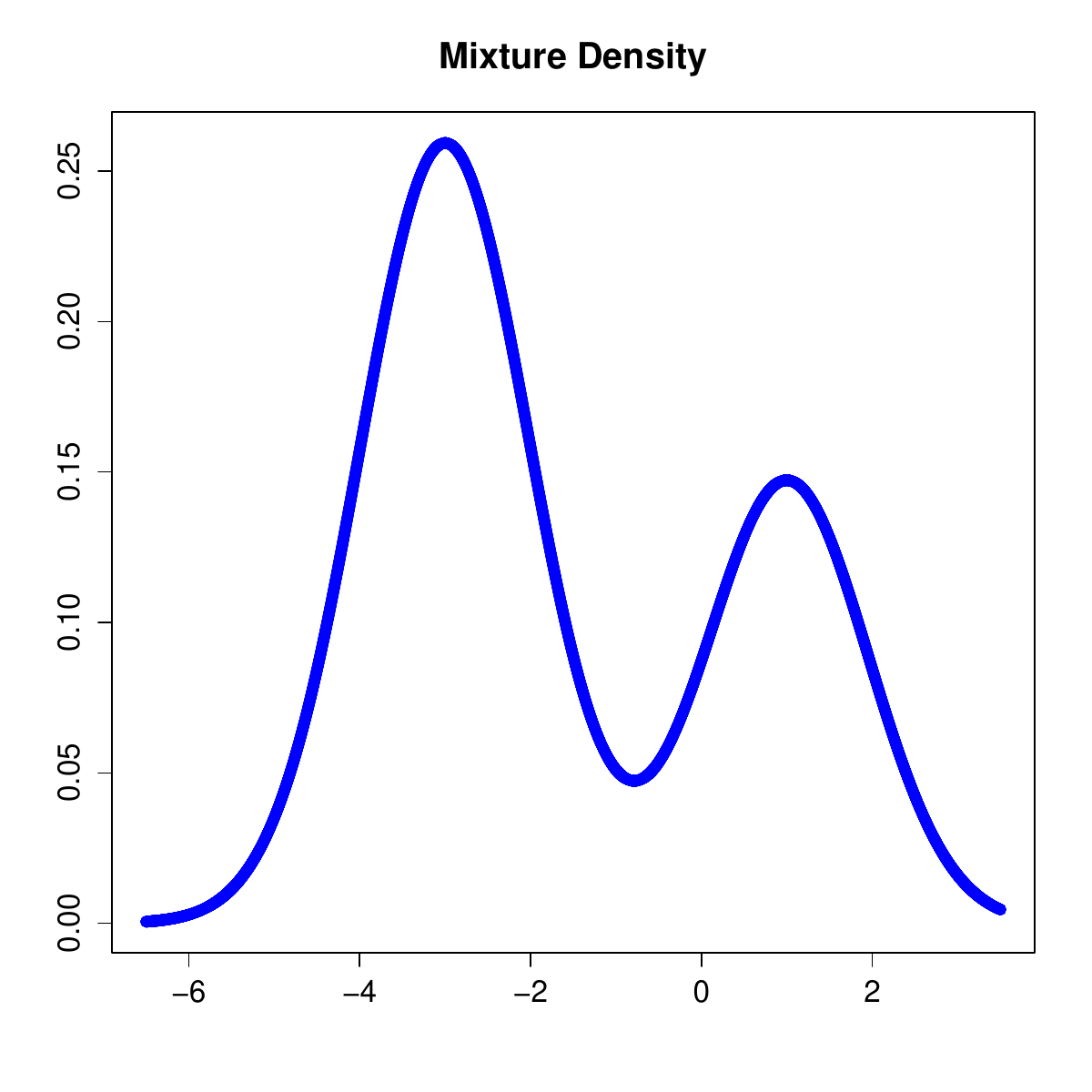}
\vspace{-0.1in}
  \caption{Mixture of two univariate Gaussians with 2 modes}
  \label{fig:mixtuniv}
\end{figure}

\begin{remark}
Since a single Gaussian has a unique global maximum at its mean $\mu$, we have that $m(d,1)=1$ for all $d \geq 1$.
\end{remark}

Our goal is to look for lower and upper bounds for $m(d,k)$. Our main results in this direction are Theorem \ref{thm:many} for the former and Corollary \ref{cor:upp} for the latter.

\section{Background}
\label{sec:two}
As stated in the introduction, a single Gaussian has a unique mode,
that is, $m(d,1)=1$ for all $d \geq 1$. The simplest case when there
is an actual mixture has $d=1$ and $k=2$: a mixture of two univariate
Gaussians $X_1 \sim N(\mu_1, \sigma_1^2)$ and $X_2 \sim
N(\mu_2,\sigma_2^2)$, with mixture parameter $\alpha \in (0,1)$. It
was observed historically that in this scenario the number of modes
was either $1$ or $2$, with the following heuristics:
\begin{itemize}
\item If the distance between the component means is small, then the
  mixture is unimodal (independently of $\alpha$).
\item If the distance between the component means is large enough,
  then there is bimodality unless $\alpha$ is close to $0$ or $1$.
\end{itemize}
A.C. Cohen (1953) and Eisenberger (1964) obtained some first explicit
conditions in these directions \cite{Beh}. Notably, if $\alpha =
\frac{1}{2}$ and $\sigma_1 = \sigma_2 = \sigma$, then the mixture is
unimodal (with mode at $\frac{\mu_1 + \mu_2}{2}$) if and only if
$\vert \mu_2 - \mu_1 \vert \leq 2 \sigma$. A few years later,
J. Behboodian gave a proof that indeed $m(1,2)=2$ by showing the
number of critical points of the density is at most three, and finds that $$\vert \mu_2
- \mu_1 \vert \leq 2 \min( \sigma_1, \sigma_2)$$ is a sufficient
condition for unimodality. Furthermore, if $\sigma_1 = \sigma_2 =
\sigma$, then $$\vert \mu_2 - \mu_1 \vert \leq 2 \sigma \sqrt{1+
  \frac{\vert \log(\alpha) - \log(1-\alpha) \vert }{2}} $$ is again a
sufficient condition for having only one mode in the mixture.
Starting the 21st century, it was Carreira-Perpi\~n\'an and Williams
who had particular interest in the problem \cite{CPW1}. Using
scale-space theory, they prove that $m(1,k)=k$; any univariate
Gaussian mixture with $k$ components has at most $k$ modes.
A natural conjecture could be that $m(d,k)=k$ for all $d,k$, that is, a mixture with $k$ Gaussian components can have at most $k$ modes. \footnote{In June 2016, a discussion thread on the {\tt ANZstat}
  mailing list (e-mail bulletin board for statistics in Australia and
  New Zealand) with the title ``\textit{an interesting
    counter-intuitive fact}" referred to the fact that a Gaussian
  mixture can have more modes than components.}
However, this fails already when $d=k=2$, since a mixture of two
bivariate Gaussians can have three distinct modes (and actually,
$m(2,2)=3$).
\begin{example}
Consider $X_1 \sim N \left( \begin{pmatrix} 1 \\
    0 \end{pmatrix}, \begin{pmatrix} 1 & 0 \\ 0 & 0.1 \end{pmatrix}
\right) $ and $X_2 \sim N \left( \begin{pmatrix} 0 \\
    1 \end{pmatrix}, \begin{pmatrix} 0.1 & 0 \\ 0 & 1 \end{pmatrix}
\right) $, with $\alpha=\frac{1}{2}$. There are two modes close to the
original means at $(1,0)$ and $(0,1)$ and there is also a third mode
near the origin.  This situation is illustrated in the contour plot of Figure
\ref{fig:mixtbiv}, with the two means marked `\emph{+}' and the three
modes marked in red.
\end{example}
\begin{figure}[t]
  \centering
 \includegraphics[scale=0.5]{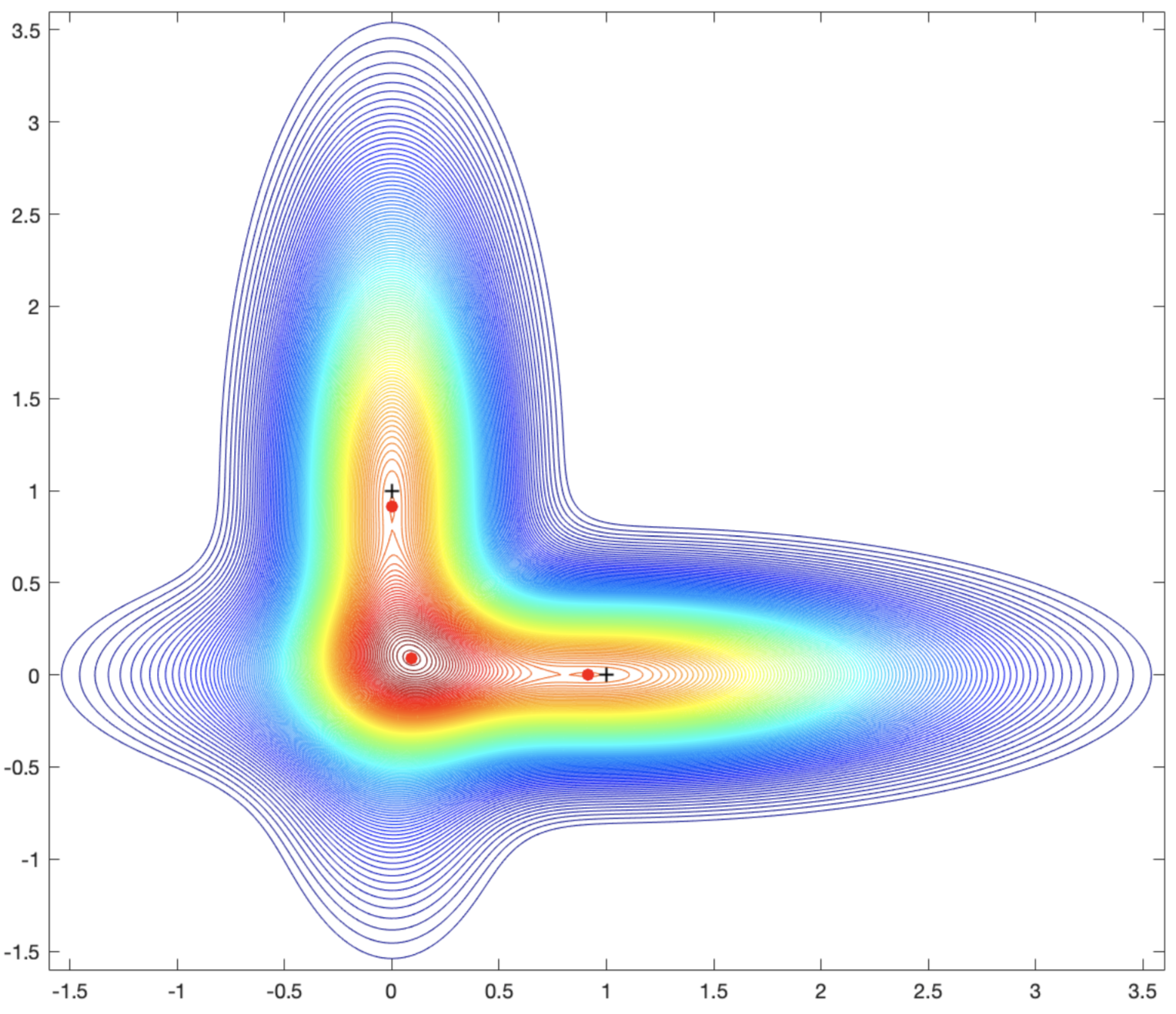}
\vspace{-0.1in}
  \caption{Mixture of two bivariate Gaussians with 3 modes}
  \label{fig:mixtbiv}
\end{figure}
Special attention can be paid to assumptions on the variances. A mixture is said to be \textit{homoscedastic} if all the variances in the components are equal: $\Sigma_i = \Sigma$ for $i=1,\ldots, k$. On the other hand, a mixture is said to be \textit{isotropic} if $\Sigma_i = \sigma_i^2 I$ for $i=1,\ldots, k$, so that covariances are scalar matrices and the densities have a `spherical' shape. Note that, up to coordinate change, homoscedastic mixtures are (homoscedastic) isotropic.   
Carreira-Perpi\~n\'an and Williams conjectured in \cite{CPW1} that if one is restricted to homoscedastic Gaussian mixtures, then the maximum number of modes is actually $k$, and verified this numerically for many examples in a brute force search. Denoting by $h(d,k)$ this maximum number, they asserted that $h(d,k)=k$ for any $d,k \geq 1$.
\begin{remark}
It holds that $h(d,k) \leq m(d,k)$ for all $d,k$ \textup{, and} $h(d,k)$ \textup{is also the maximum number of possible modes of a Gaussian mixture with all unit covariances (by the note above and the fact that the number of modes remains invariant under affine transformations)}
\end{remark}  
However, later J.J. Duistermaat emailed the authors of \cite{CPW2} with a counterexample in dimension $d=2$ with $k=3$ isotropic components, each on the vertex of an equilateral triangle. This configuration gives 4 modes for a small window of parameters, disproving the conjecture. 
\begin{example}
Consider the isotropic mixture with components $X_1 \sim N ((1,0),\sigma^2 I_2)$, $X_2 \sim N ((-\frac{1}{2},\frac{\sqrt{3}}{2}),\sigma^2 I_2)$ and $X_3 \sim N ((-\frac{1}{2},-\frac{\sqrt{3}}{2}),\sigma^2 I_2)$ with $\sigma^2 = 0.53$ and $\alpha=(\frac{1}{3},\frac{1}{3},\frac{1}{3})$. There is a mode on each of the three line segments between the origin and the means, and there is also a fourth mode at the origin. This situation is illustrated in the contour plot of Figure \ref{fig:mixttriv}, with the three means marked `\emph{+}'  and the four modes marked as red points. 
\end{example}
\begin{figure}[t]
  \centering
 \includegraphics[scale=0.5]{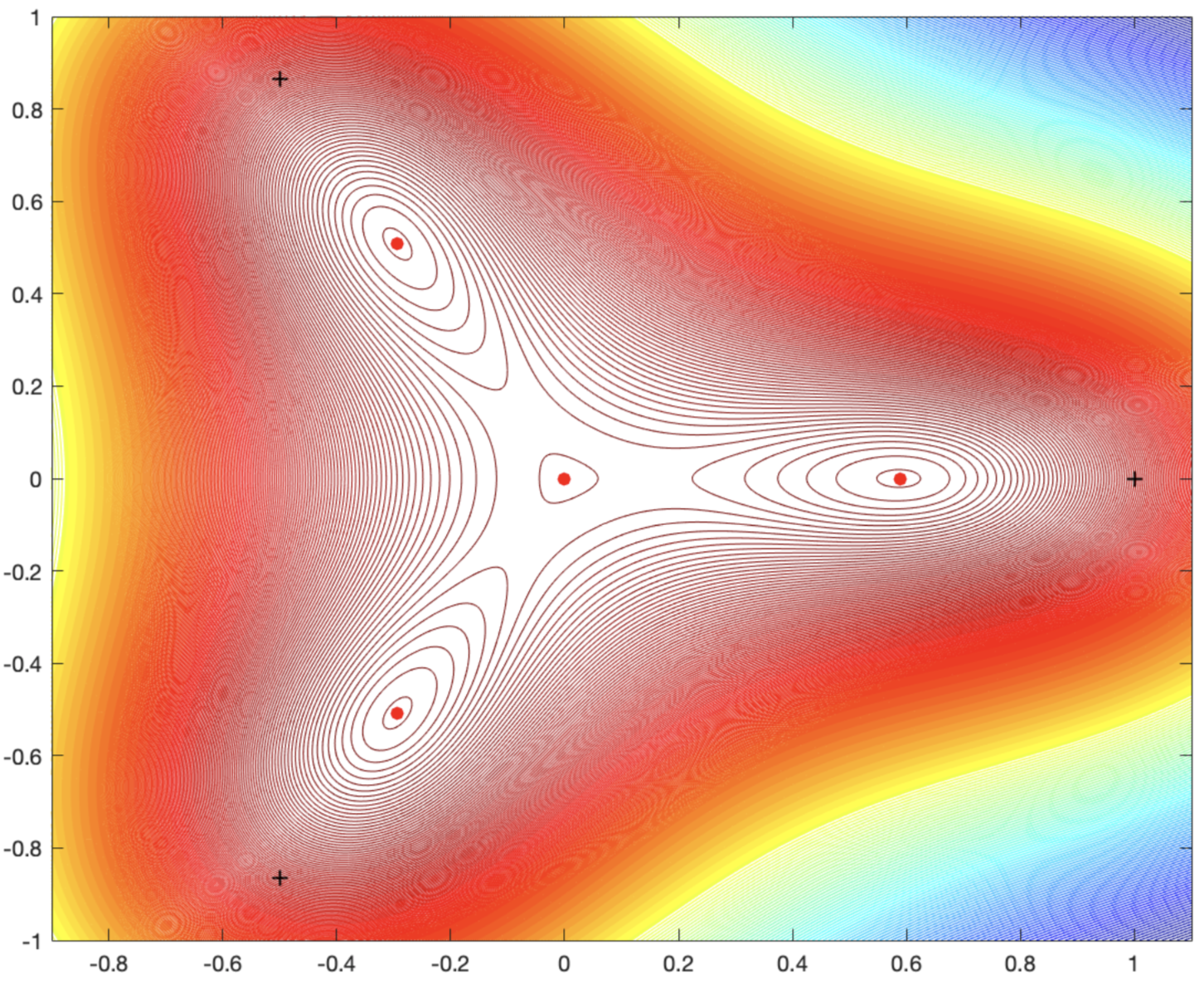}
\vspace{-0.1in}
  \caption{Duistermaat's counterexample: 4 modes in a mixture of 3 bivariate Gaussians}
  \label{fig:mixttriv}
\end{figure}
In terms of contribution to the study of the topography of Gaussian mixture densities, S. Ray and B. Lindsay initiate a systematic study and ask interesting questions in \cite{RL}. They consider the \emph{ridgeline function} $x^*: \Delta_k \rightarrow \RR^d$ given by
\begin{equation}
  x^*(\alpha) = [\alpha_1 \Sigma_1^{-1} + \alpha_2 \Sigma_2^{-1} + \ldots + \alpha_k \Sigma_k^{-1}]^{-1}[\alpha_1 \Sigma_1^{-1} \mu_1 + \alpha_2 \Sigma_2^{-1} \mu_2 + \ldots + \alpha_k \Sigma_k^{-1} \mu_k]
\end{equation}
where $\Delta_k = \{ (\alpha_1, \alpha_2, \dots, \alpha_k) | \, \alpha_i\geq 0 \text{ and } \alpha_1 + \alpha_2 + \dots + \alpha_k = 1 \}$ denotes the $(k-1)$-dimensional probability simplex, obtaining as its image the
\emph{ridgeline variety} $\mathcal{M} = \Ima(x^*)$ that contains all
critical points of $\Phi_X$ for fixed $\mu_1, \ldots, \mu_k$ and fixed
$\Sigma_1, \ldots, \Sigma_k$. This fact is useful, for example, in the
case of homoscedastic mixtures, whose critical points (and in
particular all modes) lie in the convex hull of the component means (a
result that appeared first in \cite{CPW1}). 
It would be interesting to study the locus of critical points of
the Gaussian mixture density function as the means and/or covariances vary.

In the conclusion of \cite{RL}, the following line appears:
``\textit{one might ask if there exists an upper bound for the number
  of modes, one that can be described as a function of $k$ and
  $d$}''.

Assuming this bound is finite, we answer this question in
the affirmative in Section \ref{sec:five}.

\section{Examples and Conjecture}
\label{sec: three}
  
The appearance of a possible extra mode in dimension $d=2$ when having $k=2$ components carries over to higher dimensions. Ray and Ren proved in \cite{RR} that $m(d,2)=d+1$. That is, one can get as many as $d+1$ modes from just a two component Gaussian mixture in dimension $d$.
Looking for further progress, Ray proposed the maximum number of modes problem for the 2011 AIM Workshop on Singular Learning Theory, organized by Steele, Sturmfels and Watanabe \cite{AIM}. The problem was discussed, and it led to the following conjecture: 
   
\begin{conjecture}{(Sturmfels, AIM 2011)} \label{conjlab} For all $d,k \geq 1$, 
\begin{equation} \label{boundAIM}
m(d,k) = \binom{d+k-1}{d}.
\end{equation}
\end{conjecture}  
This conjecture matches correctly all the known values for $m(d,k)$ so far, which we have presented. In the next section we will show that
for $d=2$ there exist Gaussian mixtures that achieve as many as
$\binom{k+1}{2}$ modes, showing that (\ref{boundAIM}) is a lower bound
on $m(2,k)$.

\begin{figure}[t]
  \centering
 \includegraphics[scale=0.5]{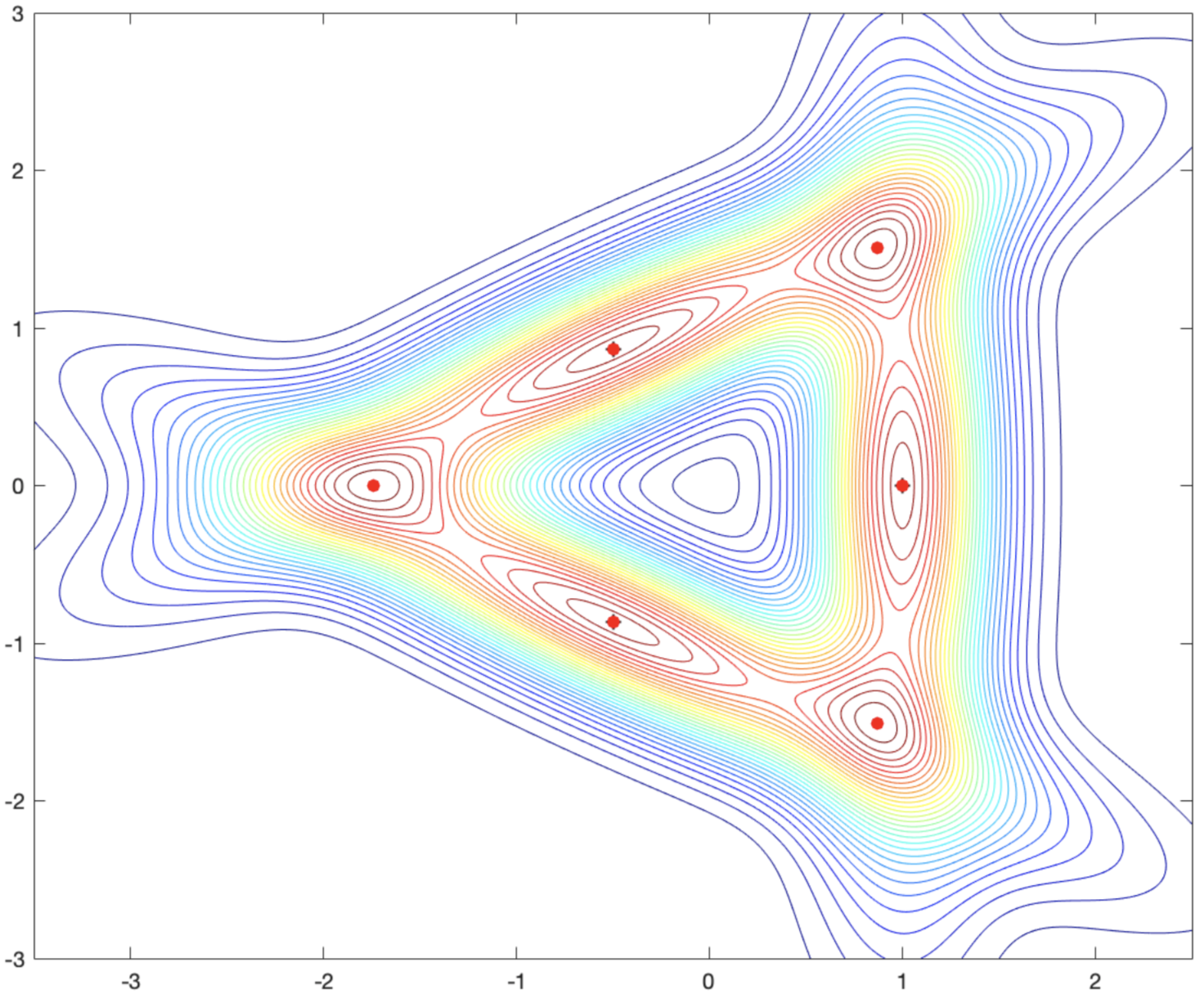}
\vspace{-0.1in}
  \caption{Mixture of 3 bivariate Gaussians with 6 modes}
  \label{fig:sixmodes}
\end{figure}

Let $d=2$ and $k=3$. The conjectured bound gives $\binom{2+3-1}{2}= 6$ modes. In Figure \ref{fig:sixmodes} we give an example of a Gaussian mixture that has this number of modes. The configuration relies on the deformation of 3 lines arranged in an equilateral triangle, and taking means as the middle points on the 3 sides, with all weights $\frac13$. Apart from the modes coming from the means, the other 3 modes lie near the corresponding triangle vertices. 

One could ask if in this anisotropic case, there exists a counterexample to Conjecture~\ref{conjlab} in the spirit of Duistermaat's mixture. Specifically, could an extra mode be formed at the origin for some values of the covariance parameters? This would give a total of 7 modes. Note that by rotational symmetry, the origin is always a critical point. Indeed, if the Gaussians are very concentrated on the lines (like in Figure
\ref{fig:sixmodes}), then the origin is a local minimum. If they diffuse enough, then it will eventually become a mode. The problem is that in this diffusion, the modes coming from the means quickly become saddle points, as illustrated in Figure \ref{fig:threemodes}.  We argue that an intermediate scenario of a total of 7 modes is actually impossible. Indeed, consider any height of the equilateral triangle. Again by symmetry, the corresponding modes near the vertex and middle point of the triangle lie on this height. Restricting the Gaussian mixture density to that line, the components corresponding to the opposite sides project to the same kernel; thus obtaining a combination of two Gaussian kernels. Since we know that the number of modes is at most two in one dimension, not all three of the critical points lying on the line can be modes.

\begin{figure}[t]
  \centering
 \includegraphics[scale=0.5]{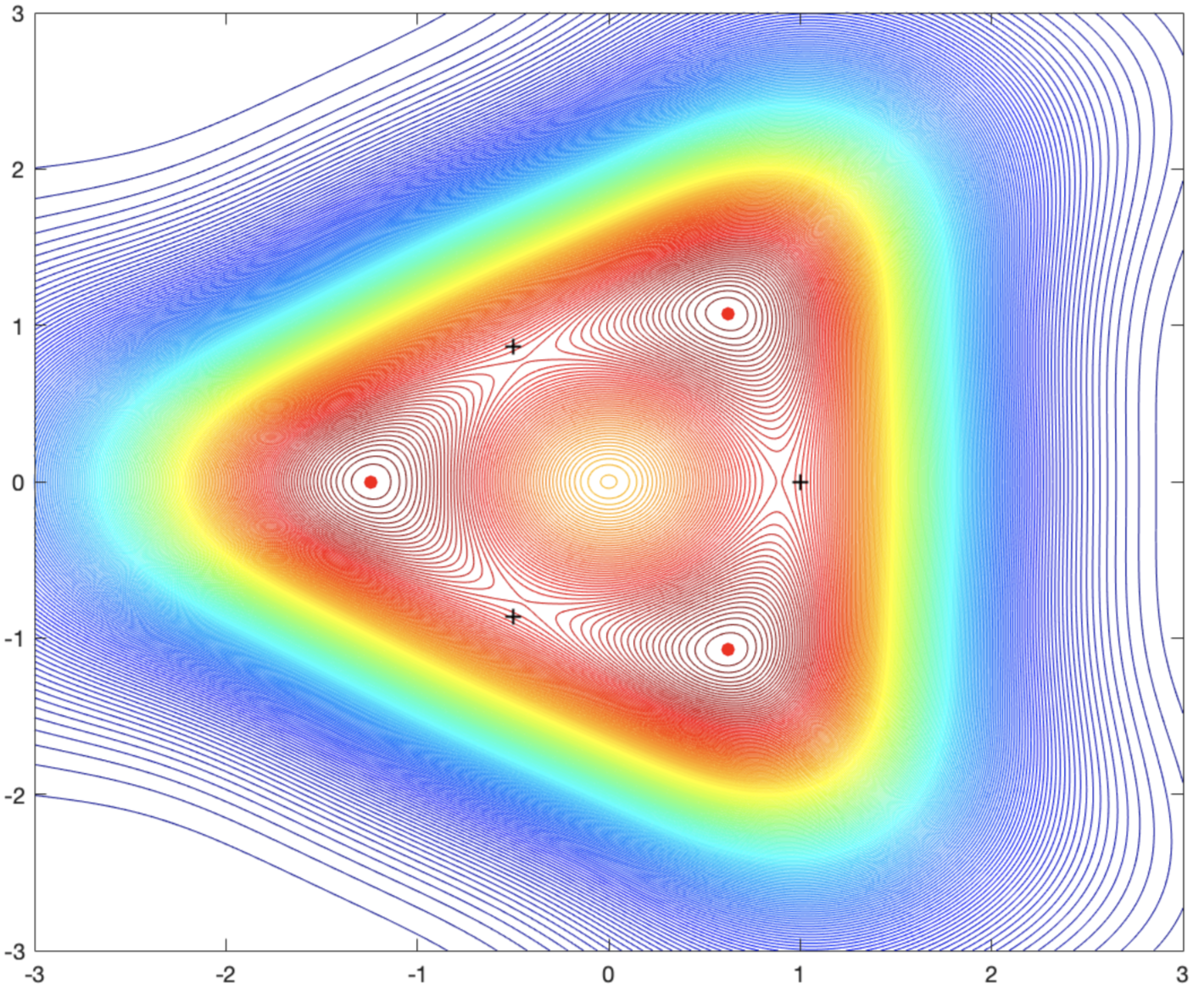}
\vspace{-0.1in}
  \caption{Mixture of 3 bivariate Gaussians with 3 modes}
  \label{fig:threemodes}
\end{figure}

In \cite{EFR}, a construction of an isotropic (and homoscedastic)
mixture of Gaussians is presented. One considers products of
triangles, using Duistermaat's counterexample with $4$ modes of
Section~\ref{sec:two} as the basic building block to obtain
$4^n$ modes in dimension $d=2n$ with $k=3^n$ components. This gives an example where the number of modes is superlinear $k^{1.261}$ in the
number of components (however, note that the dimension $d=2 \log_3 k$
also grows with $k$).

In the following section, we provide configurations for any choice of
$d>1$ and $k>1$ having $\binom{k}{d}+k$ modes. If we let $d$ grow
logarithmically with $k$ as in~\cite{EFR}, we obtain superpolynomially
(but subexponentially) many modes.

\section{Many Modes}
\label{sec:four}
\newcommand{\Fi}{F_i^\pm}
\newcommand{\Fj}{F_j^\pm}
In this section, we prove that Gaussian mixtures can have many modes.  
\begin{theorem} \label{thm:many}
  Given integers $k,d \ge 2$, there is a mixture of $k$ Gaussians in
  $\RR^d$ with at least $\binom{k}{d}+k$ modes. That is, $m(d,k)\geq \binom{k}{d}+k$.
\end{theorem}

These are the terms $i=1$ and $i=d$ in the expansion $\binom{d+k-1}{d}
= \sum_{i=1}^d \binom{d-1}{d-i} \binom{k}{i}$ of the conjectured bound
\eqref{boundAIM}. For $d=2$, our bound agrees with \eqref{boundAIM}.

\begin{proof}
Starting from a generic arrangement $H_1, \ldots, H_k$ of $k$ affine hyperplanes in $\RR^d$,
we are going to define a family $\Phi=\Phi^\delta$ of Gaussian
mixtures depending on a parameter $\delta > 0$. Around each of the
$\binom{k}{d}$ intersection vertices $p$ of the arrangement, we
construct neighborhoods $Q=Q(p)$, also depending on $\delta$,
so that for $\delta$ small enough, we have $\Phi^\delta|_{\partial Q}
< \Phi^\delta(p)$. This certifies the existence of a mode in $Q$ for each $p$ (for this, we may assume that $k \geq d$ since otherwise there are no vertices).
In addition, there will be a mode near each of the $k$ means.
For each $i = 1, \ldots, k$, denote by $\pi_i \colon \RR^d
\to H_i$ the orthogonal projection, and pick an affine map $\eta_i
\colon \RR^d \to \RR$ such that $|\eta_i(x)|$ is the distance from $x$
to $H_i$. Further, choose means $\mu_i \in H_i$ outside the other
$H_j$. Then, our $i$th component will be a standard Gaussian with
mean $\mu_i$ along $H_i$ with variance $\delta^3$ in the direction
normal to $H_i$:
\begin{equation}
  \label{eq:Phii}
  \Phi_i(x) := \frac{1}{\sqrt{(2\pi)^d}  \sqrt{\delta^3}} \ \exp \left( -
    \frac{1}{2 \, \delta^3} \, |\eta_i(x)|^2 \ - \frac12 \, \|\pi_i(x)-\mu_i\|^2
    \right)
\end{equation}
For the mixture, we take all coefficients to be equal: $\Phi = \frac1k
\sum_i \Phi_i$.
Let $p$ be one of the intersection vertices; without loss of generality, $\left\lbrace p \right\rbrace = H_1 \cap \cdots \cap H_d $. For $\delta > 0$, we define the neighborhood $Q$ of $p$ to be
$$Q(p) = \left\lbrace x \in \RR^d \, \vline \,  | \eta_i(x) | \leq \delta , \, \forall 1 \leq i \leq d \right\rbrace $$
(note that $\eta_i(p) = 0$ for $i=1,\dots,d$).   
This is an affine cube with center $p$.
Now we consider each of its $2d$ facets $\Fi$, $1 \leq i \leq d $, where
$$\Fi = \left\lbrace x \in \RR^d \, \vline \,   \eta_i(x)  = \pm \delta , \,  | \eta_j(x) | \leq \delta \,  \, \, \forall j \neq i  \right\rbrace .$$

We will show that, around the point $p$, as $\delta \to 0$,
\begin{align}
  \sqrt{\delta^3} \ \max_{x \in \Fj} \Phi_i(x)
  & \quad\to\quad
  \phi_i
    \quad \text{for } i,j=1,\ldots,d, \ j \neq i
    \label{eq:phiiFj}
  \\
  \sqrt{\delta^3} \ \max_{x \in \Fi} \Phi_i(x)
  & \quad\to\quad
  0
    \quad \text{for } i=1,\ldots,d
    \label{eq:phiiFi}
  \\
  \sqrt{\delta^3} \ \max_{x \in Q} \Phi_i(x)
  & \quad\to\quad
    0
    \quad \text{for } i=d+1,\ldots,k
    \label{eq:phiiother}
\end{align}
where $\phi_i$ is the positive number
$$ \phi_i := \frac{1}{\sqrt{(2\pi)^d}} \ \exp \left( - \frac12 \,
  \|p-\mu_i\|^2 \right) \,. $$

To establish \eqref{eq:phiiFj}, observe that along $\Fj$ ($j \neq i$),
we have
\begin{equation*}
\sqrt{\delta^3} \ \Phi_i(x) \leq \frac{1}{\sqrt{(2\pi)^d}} \ \exp
\left( - \frac12 \, \|\pi_i(x)-\mu_i\|^2 \right) .
\end{equation*}
with equality at the center of $\Fj$.
The right hand side is a continuous function of $x$, it is
independent of $\delta$, and it evaluates to $\phi_i$ at $p$. As all
of $\Fj$ converges to $p$, we must have \eqref{eq:phiiFj}.

To establish \eqref{eq:phiiFi}, observe that along $\Fi$, we have
$$ \sqrt{ \delta^3} \ \Phi_i(x) \leq \frac{1}{\sqrt{(2\pi)^d}} \ \exp \left( -
  \frac{1}{2 \, \delta^3} \, \delta^2 \ 
\right ) \quad \xrightarrow {\delta \to 0} \quad
0. $$

To establish \eqref{eq:phiiother}, fix $i \in \{d+1, \ldots, k\}$.
Observe that $|\eta_i(p)| > 0$. As the diameter of $Q$ is linear in
$\delta$, for small enough $\delta$, we have $|\eta_i(x)| > \frac12
|\eta_i(p)|$ for all $x \in Q$. Hence, for those $\delta$'s
$$
 \sqrt{ \delta^3} \ \Phi_i(x) \leq \frac{1}{\sqrt{(2\pi)^d}} \ \exp
 \left( - \frac{1}{8 \delta^3} |\eta_i(p)|^2 \right)
 \quad \xrightarrow {\delta \to 0} \quad 0. 
$$

Adding up \eqref{eq:phiiFj}, \eqref{eq:phiiFi} and \eqref{eq:phiiother}, we get for the mixture density $\Phi = \frac1k \sum_{i=1}^k
\Phi_i$ that
\begin{align}
  \sqrt{\delta^3} \ \max_{x \in \partial Q} \Phi(x)
  & \quad\to\quad
  \max_{j=1,\ldots,d} \ \frac1k \sum_{\substack{%
      i=1 \\ i \neq j}}^d \phi_i \label{eq:limit:boundary} \\
  \intertext{On the other hand, since $\Phi_i(p)=\phi_i$ for $i=1,\ldots,d$ and again by \eqref{eq:phiiother},} 
  \sqrt{\delta^3} \ \Phi(p)
  & \quad\to\quad
  \frac1k \sum_{i=1}^d \phi_i \label{eq:limit:at-p}
\end{align}

As the limit~\eqref{eq:limit:boundary} is
smaller (by one $\phi_j>0$) than the limit~\eqref{eq:limit:at-p}, there
must be some $\delta^\star(p) >0$ so that for $0 < \delta <
\delta^\star(p)$ we have $\max_{x \in \partial Q} \Phi(x) <
\Phi(p)$. Then the point $p'$ where the continuous function $\Phi$
takes it maximum over the compact set $Q$ will be in the interior of
$Q$, and hence is a local maximum.
Choosing $\delta^\star$ to be the minimum over the $\delta^\star(p)$
over all intersection vertices $p$, we obtain a mixture with at least
$\binom{k}{d}$ modes. 

The argument for the existence of a mode near $\mu_i$ is similar, but
much simpler. Fix a compact neighborhood $Q$ of $\mu_i$ which avoids
the hyperplanes $H_j$ for $j \neq i$. For small $\delta$, as in \eqref{eq:phiiother}, the $\Phi_j$
for $j \neq i$ become negligible along $Q$. Since the remaining $\Phi_i$ attains its maximum at $\mu_i$, the value of the mixture density $\Phi$ at $\mu_i$ will be larger than its values along $\partial Q$. Thus we obtain the existence of another $k$ modes for sufficiently small $\delta$.
\end{proof}
\section{Not Too Many Modes}
\label{sec:five}
The  main result of this section is to present an upper bound on the number of modes of a Gaussian mixture. We start by looking at the set of critical points and we will use Khovanskii's theory on \textit{fewnomials}, see \cite{KHO}. 
\begin{theorem} \label{uppertheorem}
For all $d,k \geq 1$, the number of non-degenerate critical points for the density of a mixture of $k$ Gaussians in $\RR^d$ is bounded by
\begin{equation} \label{upperbound}
 2^{d+ \binom{k}{2}}  (5+3d)^k.
\end{equation}
\end{theorem}
This will follow from a Khovanskii-type theorem that bounds the number of nondegenerate solutions to a system of polynomial equations that includes transcendental functions. Such a version where the transcendental functions are exponentials of linear forms was first presented by Khovanskii to illustrate his theory of fewnomials \cite[p.12]{KHO}. In our case, however, we will be interested in exponentials of quadratic forms.
\begin{theorem} \label{upperthm}
For $1\leq i \leq n$, let $F_i \in \RR[x_1,\ldots,x_n,y_1,\ldots,y_k]$ be polynomials of degree $d_i$ and for $1\leq j \leq k$ consider the exponential quadratic forms $y_j(x)=\mathrm{e}^{(x-\mu_j)^T Q_j (x-\mu_j)}$, with $\mu_j \in \RR^n$ and $Q_j \in \RR^{n \times n}$. If $g_i:\RR^n \rightarrow \RR$ are given by $g_i(x)=F_i(x_1,\ldots,x_n,y_1(x),\ldots,y_k(x))$ then the number of non-degenerate solutions to the system $g_1=g_2=\cdots=g_n=0$ is finite and bounded by 
\begin{equation} \label{fewbound}
d_1 \cdots d_n ( 5 + n + d_1 + \ldots + d_n)^k \cdot 2^{\frac{k(k-1)}{2}}. 
\end{equation} 
\end{theorem}  
In order to prove his theorem, Khovanskii gives first a sketch making simplifying assumptions (and skipping technical details) and fills the theory in his next two chapters. This sketch is also presented in \cite{SOT} and \cite{BURG}, and we will present the proof of our theorem in the same way. We will need the following lemma (for a proof see e.g. Theorem 4.3 in \cite{SOT}). 
\begin{lemma}[Khovanskii-Rolle]
Let $C \subset \RR^{n+1}$ be a smooth curve that intersects the hyperplane $H$ given by ${x_{n+1}=0}$ transversally, and $v=(v_1,\ldots,v_{n+1}): C \rightarrow \RR^{n+1}$ a smooth nonvanishing tangential vector field to $C$. 
Then $\vert (C \cap H) \vert < N + q$, where $N$ is the number of points of $C$ where $v_{n+1}=0$ and $q$ is the number of unbounded components of $C$. 
\end{lemma}

\begin{proof}  
(of Theorem \ref{upperthm}) \quad
By induction on $k$. If $k=0$, there are no exponentials and the bound
\eqref{fewbound} reduces to the product of the degrees $d_1 \cdots
d_n$. This is the well known B\'ezout bound for a multivariate system
of polynomial equations. \\ 
Now we will give the sketch of the proof and mention how our estimates
change if the smooth assumptions in the induction step do not hold. In any
case, the final inequalities needed to prove the bound
\eqref{fewbound} will hold. \\ 
For $k\geq 1$, to reduce the number of exponentials to $k-1$, we
introduce a new variable $t$ such that the system with equations
\begin{equation} \label{systcurve}
\widehat{g_i}(x,t) :=
F_i(x_1,\ldots,x_n,y_1(x),\ldots,y_{k-1}(x),ty_k(x))  
\end{equation}
has the same as the original system when
intersecting with the hyperplane $t=1$. We assume the functions
\eqref{systcurve} have $0$ as a regular value so that the locus
$\widehat{g_1}= \ldots = \widehat{g_n} = 0$ is a smooth curve $C$ in
$\RR^{n+1}$  (this is a critical step that needs to be modified
later), so we can apply the Khovanskii-Rolle Lemma. Indeed, by
Cramer's rule, the vector field
\begin{equation}
v_r := (-1)^{n+1-r} \det \frac{\partial ( \widehat{g_1}, \ldots, 
  \widehat{g_n} )}{\partial (x_1, \ldots, x_{r-1}, 
  x_{r+1}, \ldots, x_n, t) }
\end{equation}
is orthogonal to $\nabla\hat{g}_i$ for all $i$. So it is tangential to
$C$ and non-vanishing because $0$ is a regular value.
Thus, the bound for $N$ is the number of solutions to the system in
the $n+1$ variables $x_1,x_2,\ldots,x_n,u$ with $u=t y_k(x)$ and $k-1$
exponentials
\begin{equation} \label{Nsystem}
\widehat{g_1}=0, \, \ldots, \, \widehat{g_n}=0, \, \, v_{n+1}= \det \frac{\partial
  \hat{g}}{\partial x} = 0.
\end{equation}
Now $\frac{\partial y_j}{\partial x_i}(x)=y_j \cdot l_{ij}(x)$ where
$l_{ij}$ is a linear function. Hence, $\frac{\partial \widehat{g_i}
}{\partial x_j}=h_{ij}(x,y_1(x),\ldots,y_{k-1}(x),u)$ where $h_{ij}$
is a polynomial of degree at most $d_i+1$. Thus $v_{n+1}(x,u)$ is a
polynomial of degree at most $(d_1+1)+\ldots +(d_n+1)= n + D$, where
$D=d_1 + \ldots + d_n$. By induction hypothesis, 
\begin{equation}
N \leq d_1 \cdots d_n (n + D)(5 + (n+1) + (D+ n + D))^{k-1} \cdot 2^{\frac{(k-1)(k-2)}{2}}
\end{equation}
In order to bound $q$, the number of unbounded components of $C$, one
observes that 
a hyperplane sufficiently far from the
origin will meet $C$ in at least $q$ points (cf. \cite[Lemma 12.6]{BURG}). In other words, $q$ can be bounded by the number of
solutions of a system 
\begin{equation}
\widehat{g_1}=0, \, \, \ldots \, \, ,  \widehat{g_n}=0, \, \, \, \lambda_1 x_1 + \ldots +\lambda_n x_n + \lambda_{n+1}u + \mu =0,
\end{equation}
for some $\lambda_i, \mu \in \RR$. Under non-degeneracy of the solutions, we get by induction hypothesis,
\begin{equation}
q \leq d_1 \cdots d_n \cdot 1 \cdot (5+ (n+1) + (D+1))^{k-1} \cdot 2^{\frac{(k-1)(k-2)}{2}}
\end{equation}
So, in total,
\[ 
\begin{array}{rcl}
N   +  q & \leq & d_1 \cdots d_n \left[ (n+D)(6 + 2n + 2D)^{k-1} + (7
  + n + D)^{k-1} \right] \cdot 2^{\frac{(k-1)(k-2)}{2}} \\
& < &  d_1 \cdots d_n \left[ (n+D)(5 + n + D)^{k-1} \cdot  2^{k-1} +
  5(5 + n + D)^{k-1} \cdot 2^{k-1} \right] \cdot
2^{\frac{(k-1)(k-2)}{2}} \\
& = & d_1 \cdots d_n (5+n+D)^k \cdot 2^{\frac{k(k-1)}{2}}.
\end{array}
 \]
as we wanted. This is the end of the sketch.

If the smoothness assumptions for the system after introducing $t$ are
not satisfied, the argument is modified via the Morse-Sard
Theorem. The details of such modifications can be found along
\cite{KHO}, although we find that for our theorem these are better
summarized in \cite[p. 293-295]{BURG}. Essentially, one slightly
perturbs the system from $\hat{g_i}=0$ to $\hat{g_i}=\epsilon_i$
($\epsilon$ in a neighborhood of $0$) to guarantee obtaining a smooth
curve $C$.
The asserted bound \eqref{fewbound} remains unchanged, and the number
of non-degenerate solutions of the perturbed system cannot be less
than the number for the original system. 

Another change is that since the polynomial system might not
define a proper map, one adds an extra variable $x_0$ with an extra
equation
\begin{equation} \label{extraeq}
g_{0}(x_0,x_1,\ldots,x_n)=x_0^2 + x_1^2 + \ldots + x_n^2 - R^2 = 0
\end{equation}
with $R>0$ so that every preimage is now bounded. Morse-Sard now applies to conclude the set of regular values of $g=(g_0,\ldots,g_n): \RR^{n+1} \rightarrow \RR^{n+1}$ is open and dense. 
The number of non-degenerate solutions of the new system has twice the number of non-degenerate solutions of the original system that lie in the open ball of radius $R$ centered at the origin (because from \eqref{extraeq}, a solution $(x_1,x_2,\ldots,x_n)$ gives two possible values for $x_0$). In terms of bounding the corresponding $N',q'$, we now have
\begin{equation}
N' \leq 2 \cdot d_1 \cdots d_n \cdot (n + 1 + D)(5 + (n+2) + (2 +D + n+1 +D))^{k-1} \cdot 2^{\frac{(k-1)(k-2)}{2}} 
\end{equation}
(the extra 2 comes from the degree of \eqref{extraeq}, and we now have $n+2$ variables). For $q'$, the bound becomes
\begin{equation}
q' \leq 2 \cdot d_1 \cdots d_n \cdot 1 \cdot ( 5 + (n+2) + 2 + D+1)^{k-1} \cdot 2^{\frac{(k-1)(k-2)}{2}}
\end{equation}
(the extra 1 from the hyperplane equation).
Since the $R>0$ does not affect the bound computation, it can be taken large enough to include all the solutions to the original system. Thus $N'+q'$ is a bound for twice as many non-degenerate solutions of said original system. Finally, this way the induction step inequality can again be completed 
\[ 
\begin{array}{rcl}
\dfrac{N' + q'}{2} & \leq & d_1 \cdots d_n \left[ (1+n+D)(10 + 2n + 2D)^{k-1} + (10 + n + D)^{k-1} \right] \cdot 2^{\frac{(k-1)(k-2)}{2}} \\
& < &  d_1 \cdots d_n \left[ (1+n+D)(5 + n + D)^{k-1} \cdot  2^{k-1} + 4(5 + n + D)^{k-1} \cdot 2^{k-1} \right] \cdot 2^{\frac{(k-1)(k-2)}{2}} \\
& = & d_1 \cdots d_n (5+n+D)^k \cdot 2^{\frac{k(k-1)}{2}},
\end{array}
 \]
as needed. 
\end{proof}
Now we can obtain Theorem \ref{uppertheorem} as a corollary of the above.
\begin{proof}(of Theorem \ref{uppertheorem})
Let $\Phi(x)= \sum_{i=1}^k \alpha_i \Phi_i(x)$ be a Gaussian mixture
and consider the system $g_i(x_1,\ldots,x_d)$ given by the partial
derivatives $g_i= \frac{\partial \Phi}{\partial x_i}$. These can be
interpreted as polynomials $F_i(x_1,\ldots,x_d,y_1,\ldots,y_k)$ by
taking $y_i$ as the exponential kernel of $\Phi_i$. The system now
has the form as in Theorem \ref{upperthm}, and note that the degree of
each $F_i$ is $2$.  The number of non-degenerate critical points of
$\Phi$ is thus the number of non-degenerate solutions to the system of
$g_i$, and according to (\ref{fewbound}), it is bounded by
$$2 \cdots 2 \,( 5 + d + 2 + \ldots + 2)^k \cdot 2^{\frac{k(k-1)}{2}}
= 2^d (5+d+2d)^k \,  2^{\binom{k}{2}} =  2^{d+ \binom{k}{2}}
(5+3d)^k. $$
\end{proof}

Finally, we show as promised that \eqref{upperbound} is an upper bound on the number of modes of a Gaussian mixture, provided it is finite.

\begin{theorem} \label{thm:upper}
  If a mixture of $k$ Gaussians in $\RR^d$ has
  finitely many modes, then their number is bounded by
\begin{equation} 
2^{d+ \binom{k}{2}}  (5+3d)^k.
\end{equation}
\end{theorem}

\begin{proof}
Let $\Phi(x)= \sum_{i=1}^k \alpha_i \Phi_i(x)$ be the pdf of the Gaussian mixture. If all of its modes are non-degenerate, then Theorem \ref{upperthm} applies and we're done. The difficulty stems from considering possible degenerate modes. Note that by the finiteness hypothesis, all of them are isolated so we may fix disjoint neighborhoods $Q_i \subset \RR^d$ over which each mode is the unique global maximum.

For any linear function $\ell(x_1,\ldots,x_d)=c \cdot x$, the function $\Phi + \ell$ has a gradient that differs by the constant vector $c$ from $\nabla \Phi$. In particular, the system given by the partial derivatives of $\Phi+\ell$ are still polynomials $F_i(x_1,\ldots,x_d,y_1,\ldots,y_k)$ of degree 2. By Theorem \ref{uppertheorem} we have the bound \eqref{upperbound} on the non-degenerate modes of $\Phi+\ell$. Since $\Phi: \RR^d \rightarrow \RR$ is smooth, one of Morse's Lemmas \cite[Lemma A, p.11]{MILNOR} states that for almost all $c \in \mathbb{R}^d$ (all except for a set of measure zero), $\Phi + \ell$ has only non-degenerate critical points. 

Now, take $c \in \mathbb{R}^d$ in the complement of such measure zero set, with norm small enough so that $\Phi+\ell$ still has modes inside each of the neighborhoods $Q_i$. Then, $\Phi+\ell$ may have fewer critical points than $\Phi$ but we know $\Phi+\ell$ has at least as many modes as $\Phi$ does. Since the modes of $\Phi+\ell$ are bounded by \eqref{upperbound}, the result follows.
\end{proof}

\begin{corollary} \label{cor:upp}
If every mixture of $k$ Gaussians in $\RR^d$ has finitely many modes, then
 $$m(d,k) \leq 2^{d+ \binom{k}{2}}  (5+3d)^k.$$
\end{corollary}

\section{Conclusion and Future Work}
\label{sec:six}
One of the motivations to study critical points of Gaussian mixtures
comes from the \textit{mean shift algorithm}. It converges if there
are only finitely many critical points.
This was the main goal sought in \cite{W} for Gaussian kernels.
As our bound \eqref{fewbound}
only bounds the number of non-degenerate critical points of Gaussian mixtures, a final
answer is still open. However, since non-degenerate critical points are isolated, we see that the set of
critical points of a Gaussian mixture is
finite if and only if it consists of isolated points, since this set is closed and bounded
(compare~\cite{GHA}).

Quantitatively, we do not expect our upper bound to be tight.
Rather, proving the lower bound $\binom{d+k-1}{d} $ for all $d,k$ will
be the main focus of a forthcoming paper, extending the technique used to prove Theorem \ref{thm:many}.

We observe that the construction strategies used in our lower bound can be extended to elliptical distributions, not only Gaussians, and this could be pursued further. For example, an extension of Ray and Lindsay's concept of \textit{ridgeline} for mixtures of elliptical distributions is done in \cite{AHR}, including a study of modes for mixtures of two $t$-distributions.

\section*{Acknowledgements}
The authors would like to thank Bernd Sturmfels for bringing the
problem to their attention and for his supportive comments. Thanks to
Peter Green for pointing to the discussion thread in the
{\tt ANZstat} mailing list.  Carlos Am\'endola was supported by the Einstein Foundation Berlin.  Alexander Engstr\"om would like to thank G\"unter Ziegler and Freie Universit\"at Berlin for their hospitality during his sabbatical. The authors also express their gratitude to anonymous referees for their helpful suggestions.

\end{document}